\documentclass[12pt]{amsart}
\usepackage{amsmath,amssymb,amscd,graphicx}
\usepackage[colorlinks=true, pdfstartview=FitV, linkcolor=blue, citecolor=blue, urlcolor=blue]{hyperref}
\usepackage{mathpazo}

\theoremstyle{plain}
\newtheorem{theorem}{Theorem}
\newtheorem*{conjecture}{Conjecture}
\newtheorem*{sconjecture}{Summary Conjecture}
\newtheorem{lemma}{Lemma}

\newtheorem{proposition}{Proposition}

\theoremstyle{remark}

\DeclareMathOperator{\csch}{csch}

\newcommand{\tre}{\text{Re}}
\newcommand{\tim}{\text{Im}}
\newcommand{\pp}{{\prime\prime}}

\theoremstyle{definition}
\newtheorem*{hypothesis}{Hypothesis D}

\begin{document}
\title[Zhang's Eta Function]{Level curves for Zhang's eta function}
\author{Jeffrey Stopple}
\address{Department of Mathematics\\University of California, Santa Barbara\\Santa Barbara, CA 93106-3080 USA}
\begin{abstract}
Study of the level curve $\tre(\eta(s))=0$ for $\eta(s)=\pi^{-s/2}\Gamma(s/2)\zeta^\prime(s)$ gives a new classification of the zeros of $\zeta(s)$ and of $\zeta^\prime(s)$.  We  conjecture that for  type 2 zeros,  $\liminf (\beta^\prime -1/2)\log\gamma^\prime = 0 \Leftrightarrow \liminf (\gamma^+-\gamma^-)\log \gamma^\prime=0$, and reduce the conjecture to a lower bound on the curvature of the level curve.  We compute and classify $10^6$ zeros of $\zeta^\prime(s)$ near $T=10^{10}$.  The Riemann Hypothesis is assumed throughout.  An appendix develops the analogous classification for characteristic polynomials of unitary matrices.
\end{abstract}
\email{stopple@math.ucsb.edu}

\maketitle

 \baselineskip=16pt

\subsection*{Introduction}

The horizontal distribution of the zeros of $\zeta^\prime$ has been studied by many authors since Levinson and Montgomery \cite{LandM}.  In \cite{Sou}, Soundararajan made the following conjecture:  Assume the Riemann Hypothesis.  Then
\[
\liminf (\beta^\prime -1/2)\log\gamma^\prime = 0 \Leftrightarrow \liminf (\gamma^+-\gamma^-)\log \gamma^-=0,
\]
where $1/2+i\gamma^-$ and $1/2+i\gamma^+$ denote consecutive zeros on the critical line, and $\beta^\prime+i\gamma^\prime$ denotes a typical zero of $\zeta^\prime(s)$.  In \cite{Zhang}, Yitang Zhang proved the $\Leftarrow$ half of the conjecture.  Partial results in the other direction include
\[
\liminf (\beta^\prime -1/2)(\log\gamma^\prime)^3 = 0 \Rightarrow \liminf (\gamma^+-\gamma^-)\log \gamma^-=0,
\]
due to Fan Ge \cite{FanGe2}.  If one also assumes the zeros are simple, results of Garaev and Yildirim in \cite{Garaev} can be interpreted to say that.
\[
\liminf (\beta^\prime -1/2)\log\gamma^\prime (\log\log\gamma^\prime)^2= 0 \Rightarrow \liminf (\gamma^+-\gamma^-)\log \gamma^-=0.
\]

In this paper, rather than change the scaling, we look instead at an infinite subset of the zeros.  The starting point is the function
\[
\eta(s)=h(s)\zeta^\prime(s),\quad\text{ where }\quad h(s)=\pi^{-s/2}\Gamma(s/2),
\]
so named by Zhang in \cite{Zhang}. This function has an interesting property with respect to the zeros of $\zeta(s)$ on the critical line:
\begin{lemma} (\cite[Lemma 1]{Zhang}) Suppose $t>7$.  Then  we have $\zeta(1/2+it)=0$  if and only if $\tre(\eta(1/2+i t))=0$.
\end{lemma}
The lemma makes the level curves for $\tre(\eta(s))=0$ of interest.   Figure \ref{F:2} shows examples, where we use color to indicate the sign of $\tim(\eta(s))$.  
Green indicates $\tre(\eta(s))=0$ and $\tim(\eta(s))>0$, while purple indicates $\tre(\eta(s))=0$ and $\tim(\eta(s))<0$.

For shorthand when referring to \lq the zeros\rq\  of $\zeta(s)$ we mean the nontrivial zeros in the upper half plane.   The Riemann zeros $\rho=1/2+i\gamma$ of $\zeta(s)$ occur where the green and purple contours cross the critical line.  The zeros $\rho^\prime$ of $\zeta^\prime$ are visible everywhere the two colors come together (exclusive of the double pole at $s=1$).

\begin{figure}
\begin{center}
\includegraphics[scale=1, viewport=0 0 350 300,clip]{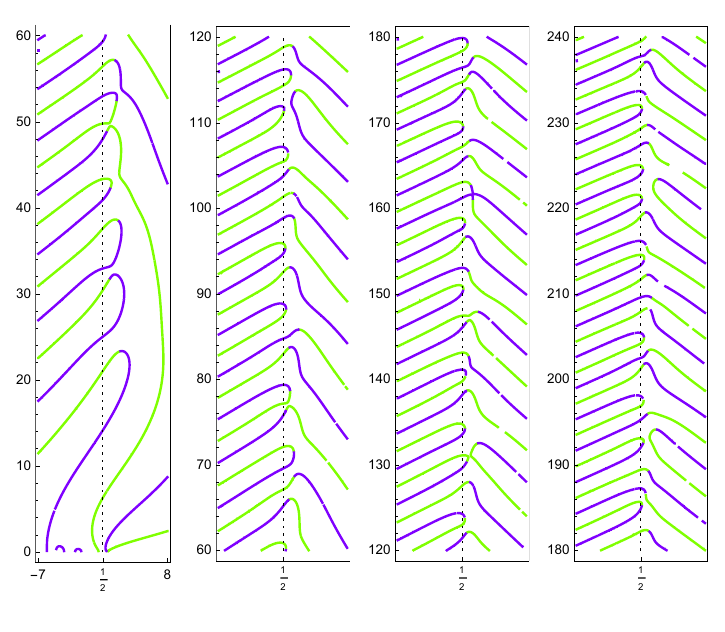}
\caption{Level curves for $\tre(\eta(s))=0$, $-7\le\sigma\le 8$, $0\le t\le 240$}\label{F:2}
\end{center}
\end{figure}

Throughout we assume the Riemann Hypothesis, in the strong form that the nontrivial zeros are also simple.
We also need to assume the following:
\begin{hypothesis}   The level curves $\tre(\eta(s))=0$ are differentiable.  This is automatic except at isolated points where $\eta^\prime(s)=0$, so we are really assuming that when $\eta^\prime(s)=0$, $\arg(\eta(s))\ne\pm\pi/2$.  This prevents the level curves from branching.   Hypothesis D is  plausible because $\pm\pi/2$ are only two points on the unit circle, while the zeros of $\eta^\prime(s)$ form a countable set.
\end{hypothesis}

\noindent
Here's a summary of the sections of the paper:
\begin{enumerate}
\item[\S\ref{S:classify}] Classification of the zeros of $\zeta(s)$ and $\zeta^\prime(s)$ into different types by means of the level curves, and results on the asymptotics of the types.  
\item[\S\ref{S:data}] Computation and classification of  $10^6$ zeros of $\zeta^\prime(s)$ near $T=10^{10}$.
\item[\S\ref{S:lemmas}] Two lemmas.
\item[\S\ref{S:Type2}] A closer look at the type 2 zeros, and the curvature of the level curve.
\item[\S\ref{S:Spira}] A canonical bijection between the complex zeros $\rho^\prime$ of $\zeta^\prime(s)$, and the complex zeros $\rho^\pp$ of $\zeta^{\prime\prime}(s)$.
\item[\S\ref{S:Marden}]  Adaptation of a theorem of Marden, and the location of $\rho^\pp$ relative to $\rho^\prime$.
\item[\S\ref{S:application}] Curvature of the level curve at $\rho^\prime$ in terms of all the other zeros $\lambda^\prime\ne\rho^\prime$.
\item[\S\ref{S:Appendix}]  Appendix: With $p_A(z)$ the characteristic polynomial of a unitary matrix $A$, we give a classification of the zeros of $p_A(z)$ and $p_A^\prime(z)$ analogous to that in \S\ref{S:classify}.
\end{enumerate}

\section{Classification of zeros}\label{S:classify}

\begin{proposition} With the usual indexing $\gamma_1<\gamma_2<\ldots$ of the imaginary parts of the zeros of $\zeta(s)$, every odd indexed zero lies on a contour $\tim(\eta(s))<0$.  Every even indexed zero lies on a contour $\tim(\eta(s))>0$. 
\end{proposition}
\begin{proof}
This follows from Lemma \ref{Lemma:ZhangGe} below, which says that as $t$ increases, the argument of $\eta(1/2+it)$ decreases by exactly $\pi$ between consecutive zeros.  A \emph{Mathematica} calculation of $\eta(1/2+i\gamma_1)$ determines the parity of all the zeros.
\end{proof}

\subsubsection*{Zeros of $\zeta^\prime(s)$}
\begin{itemize}
\item[Type 0:] We will say a zero $\rho^\prime$ of $\zeta^\prime(s)$ is of \textsc{type 0} if \emph{neither} of the level curves $\tre(\eta(s))=0, \tim(\eta(s))>0$ and $\tre(\eta(s))=0, \tim(\eta(s))<0$ exiting $\rho^\prime$ cross the critical line $\sigma=1/2$.
\item[Type 1:] We will say a zero $\rho^\prime$ of $\zeta^\prime(s)$ is of \textsc{type 1} if \emph{exactly one} of the level curves $\tre(\eta(s))=0, \tim(\eta(s))>0$ and $\tre(\eta(s))=0, \tim(\eta(s))<0$ exiting $\rho^\prime$ crosses the critical line $\sigma=1/2$.
\item[Type 2:] We will say a zero $\rho^\prime$ of $\zeta^\prime(s)$ is of \textsc{type 2} if the level curves $\tre(\eta(s))=0, \tim(\eta(s))>0$ and $\tre(\eta(s))=0, \tim(\eta(s))<0$ exiting $\rho^\prime$ \emph{both} cross the critical line $\sigma=1/2$.
\end{itemize}

\subsubsection*{Zeros of $\zeta(s)$}
\begin{itemize}
\item[Type 1:] We will say a zero $\rho=1/2+i\gamma$ of $\zeta(s)$ is of \textsc{type 1} if the level curve $\tre(\eta(s))=0$ on which it lies, terminates in a zero $\rho^\prime$ which is of type 1.
\item[Type 2:] We will say a zero $\rho=1/2+i\gamma$ of $\zeta(s)$ is of \textsc{type 2} if the level curve $\tre(\eta(s))=0$ on which it lies, terminates in a zero $\rho^\prime$ which is of type 2.
\end{itemize}

In Figure \ref{F:2}, when both branches form a loop to the left, it is type 2.  When they loop to the right, it is type 0.  If the two colors extend in opposite directions without looping, it is type 1.  In Figure \ref{F:2}, the first four zeros of $\zeta^\prime(s)$ have type 2; the next four alternate between types 1 and 2.  The first zero of type 0 occurs at height about 113, with another at height about 132.  At height about 161 we have two consecutive zeros of type 1, but from the way the graphics are imported into Latex one can not tell, looks like it might be a type 2 and type 0.  It seems a zero of $\eta^\prime(s)$ is nearby.  

Let
\[
N_1(T)=\sharp\left\{\text{type 1 zeros }1/2+i\gamma\,|\,0<\gamma<T\right\}.
\]
NB: This is a not the traditional definition of $N_1(T)$.
Let
\[
N_2(T)=
\sharp\left\{\text{pairs of type 2 zeros }1/2+i\gamma^-,1/2+i\gamma^+\,|\,0<\gamma^+<T\right\}.
\]
For $j=0,1,2$, let
\[
N_j^\prime(T)=\sharp\left\{\text{zeros }\rho^\prime=\beta^\prime+i\gamma^\prime\text{ of type }j\,|\,0<\gamma^\prime<T\right\}.
\]

\begin{theorem}\label{Thm:classification}
Every Riemann zero is of either type 1 or type 2.  Thus we have a canonical mapping from the zeros of $\zeta(s)$ to those of $\zeta^\prime(s)$, which is two to one on the type 2 zeros, and one to one on the type 1 zeros.  Zeros of $\zeta^\prime(s)$ of type 0 are precisely those not in the image of this mapping.  The Riemann zeros of type 2 are canonically grouped in pairs.

There are infinitely many type 2 zeros of $\zeta(s)$ and of $\zeta^\prime(s)$, and in fact
\[
N_2(T)=N_2^\prime(T)\gg T.
\]
At least one of the other types of zeros of $\zeta^\prime(s)$ is infinite in number.
\end{theorem}
\begin{proof}
Regarding the mapping, all this is clear except the first statement, which says that the contours which cross the critical line from the left must terminate in exactly one zero of $\zeta^\prime(s)$.  Since we are assuming Hypothesis D, the alternatives we must rule out are continuation of the contour on to the right, or looping back to the left.  

For the first possibility, note that the contour $\arg(\eta(s))=\pi/2$ (resp. $\arg(\eta(s))=-\pi/2$) does not exist in isolation; it is part of a continuum which deform smoothly as the argument is varied. 
 By  Lemma \ref{Lemma:ZhangGe} at the end of this section, the argument is decreasing along the contour $s=1/2+it$, as $t>4$ increases.
But the argument of $\eta(s)$ is increasing along the contour $4+it$ as $t$ increases, by Lemma \ref{Lemma:ZhangGe3} and the subsequent remark.  To change the orientation, the contours must cross over each other, and  this can only happen where the argument of $\eta(s)$ is undefined, at a zero $\rho^\prime$. 
  
The second possibility is ruled out by Lemma \ref{Lemma:ZhangGe}, which says that the argument of $\eta(s)$ decreases monotonically as one moves up the critical line.

Regarding the asymptotics, we have (due to Littlewood, as we are assuming the Riemann Hypothesis) 
\begin{equation}\label{Eq:1}
N_1(T)+2N_2(T)=\frac{T}{2\pi}\log\left(\frac{T}{2\pi}\right)-\frac{T}{2\pi}+O\left(\log T/\log\log T\right).
\end{equation}
Via the mapping, $N_1(T)=N_1^\prime(T)$ and $N_2(T)=N_2^\prime(T)$.
Thus we have from \cite{FanGe1}:
\begin{equation}\label{Eq:2}
N_0^\prime(T)+N_1(T)+N_2(T)=\frac{T}{2\pi}\log\left(\frac{T}{4\pi}\right)-\frac{T}{2\pi}+O\left(\log T/\log\log T\right).
\end{equation}
Subtracting (\ref{Eq:2}) from  (\ref{Eq:1}) gives
\begin{equation}\label{Eq:3}
N_2(T)-N_0^\prime(T)=\frac{T}{2\pi}\log\left(2\right)+O\left(\log T/\log\log T\right).
\end{equation}
Subtracting (\ref{Eq:1}) from twice (\ref{Eq:2}) gives
\begin{equation}\label{Eq:4}
N_1(T)+2N_0^\prime(T)=\frac{T}{2\pi}\log\left(\frac{T}{8\pi}\right)-\frac{T}{2\pi}+O\left(\log T/\log\log T\right).
\end{equation}
\end{proof}
 
 \section{Data}\label{S:data}
 
In \emph{Mathematica} we computed and classified 1,001,390 zeros of $\zeta^\prime(s)$ near  $T=10^{10}$.  We found 239,556 zeros of type 0 (23.9\%), 488,412 zeros of type 1 (48.8\%), and 273,422 zeros of type 2 (27.3\%). 

We make the following conjecture:
\begin{conjecture}  For some constant $C$, possibly equal $0$
\begin{align*}
N_0^\prime(T)=&\frac{1}{8\pi} T\log\left(\frac{T}{4\pi}\right)+\left(C-\frac{\log 2}{4\pi} \right)\cdot T+O(\log T/\log\log T)\\
N_1(T)=&\frac{1}{4\pi} T\log\left(\frac{T}{4\pi}\right)-\left(2C+\frac{1}{2\pi}\right)\cdot T+O(\log T/\log\log T)\\
N_2(T)=&\frac{1}{8\pi}T\log\left(\frac{T}{4\pi}\right)+\left(C+\frac{\log 2}{4\pi} \right)\cdot T+O(\log T/\log\log T).
\end{align*}
\end{conjecture}
The coefficients of the $T\log(T)$ terms in the conjecture are based on the heuristic that each of the two contours $\tre(\eta(s))=0,$ $\tim(\eta(s))>0$ and $\tre(\eta(s))=0,$  $\tim(\eta(s))<0$ emanating from a zero $\rho^\prime$ of $\zeta^\prime$ has equal chance of exiting the critical strip to the left or to the right.  And this is in rough agreement with the numerical evidence.  The lower order terms are the simplest expression in agreement with (\ref{Eq:2}), (\ref{Eq:3}) and (\ref{Eq:4}).

\begin{figure}
\begin{center}
\includegraphics[scale=1.3, viewport=0 0 310 180,clip]{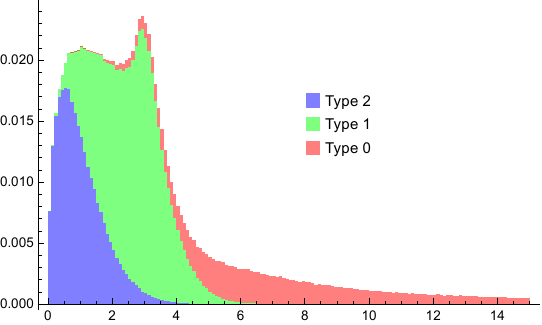}
\caption{$(\beta^\prime-1/2)\log(\gamma^\prime)$ for  $10^6$ zeros near $T=10^{10}$.}\label{F:3}
\end{center}
\end{figure}
In Figure \ref{F:3} we show the histogram of 
$
(\beta^\prime-1/2)\log(\gamma^\prime)
$
for zeros of type 0, 1, and 2 separately, for the 1,001,390 zeros $\rho^\prime$ computed.   (This is the analog of the data in Figure 5 in \cite{Duenez}, now separated by types.  See also Figure \ref{AF:3} in the Appendix)\ \   In \cite{Duenez}, the authors write \emph{\lq\lq We would like to know the underlying cause of the curious \lq second bump\rq in the distribution of zeros of derivatives [of characteristic polynomials of unitary matrices]... In Figure 5 we find a similar shape for the distribution of zeros of $\zeta^\prime$.\rq\rq}\ \  Interestingly, the histograms analogous to Figure \ref{F:3} for the three types separately each show only a single peak; it is the interplay between them that causes the second bump.

In \cite{FarmerKi}, Farmer and Ki show that if $\zeta^\prime(s)$ has sufficiently many zeros close to the
critical line, then $\zeta(s)$ has many closely spaced zeros. This gives them a condition on the zeros of
$\zeta^\prime(s)$ which implies a lower bound of the class numbers of imaginary quadratic
fields.
One sees in Figure \ref{F:3} the type 2 zeros closest to the critical line, and the type 0 zeros the furthest.
In fact the median value of $(\beta^\prime-1/2)\log(\gamma^\prime)$ for the type two zeros is $0.889$; the other quartiles are $0.487$ and $1.443$.  

The data strongly motivates the further study of the types, and in particular, the type 2 zeros.   Corresponding to the type 2 zeros of $\zeta^\prime(s)$ in the numerical data, we have 273,422 pairs of canonically associated type 2 zeros $1/2+i\gamma^-$, $1/2+i\gamma^+$ of $\zeta(s)$.
Figure \ref{F:4} shows the normalized gap 
\[
(\gamma^+-\gamma^-)\cdot \frac{\log(\gamma^\prime)}{2\pi};
\]
99.7\% are less than the average gap and 32.2\% are less than half the average gap.

\begin{figure}
\begin{center}
\includegraphics[scale=1.3, viewport=0 0 310 160,clip]{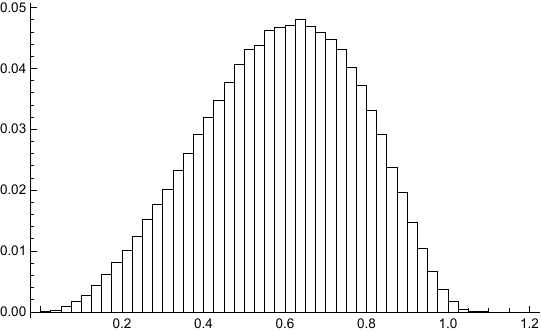}
\caption{Normalized gaps between pairs of type 2 zeros of $\zeta(s)$.}\label{F:4}
\end{center}
\end{figure}

The average value of $\beta^\prime$ is $1/2+\log\log(\gamma^\prime)/\log(\gamma^\prime)$.  The extra log log factor  probably comes from a  small number
of $\rho^\prime$ that are very far from the critical line, and  conjecturally almost all $\rho^\prime$ have real part $1/2+c/\log \gamma^\prime$ with $c$ bounded.    Spira showed in \cite{Spira2} that the number of $\rho^\prime$ with real part greater than $\alpha>1/2$ up to height $T$ is $O(T)$.  On the other hand, \cite[Theorem 11.5C]{Tit}  says there is a constant $E$, $2<E<3$ so that $\zeta^\prime(s)$ has an infinity of zeros in every strip between $\sigma=1$ and $\sigma=E$.

We can rescale by $1/\log\log(10^{10})$ to see the median for type 2 zeros on this scale is $0.283$.  For comparison, the median for type 0 zeros on this scale is $2.48$.  
Although we are primarily interested in zeros of $\zeta^\prime(s)$ close to the critical line, we briefly investigate in the rest of this section,  zeros which are far from the critical line.   Following Akatsuka \cite{Akatsuka}, let $G(s)=-2^s/\log(2)\zeta^\prime(s)$, so 
$$
\text{im}(G(s))=-\sum_{n=3}^\infty \frac{\log n}{\log 2}(2/n)^\sigma\sin(t\log(n/2)).
$$
For $\sigma\gg 1$ the first term dominates, and vanishes for $t=n\pi/\log(3/2)$.  The level curves $\text{im}(G(s))=0$ are asymptotic to horizontal lines spaced $\pi/\log(3/2)\approx 7.748$ apart.
On taking derivatives with respect to $\sigma$, with $t\approx n\pi/\log(3/2)$ one sees the real part is decreasing (as $\sigma$ increases) for even n, taking values in $(1,\infty)$, and so does not vanish.  For odd $n$ the real part is increasing, taking values in $(-\infty,1)$, and so will vanish.  This is the analog for $\zeta^\prime(s)$ of results of Arias-de-Reyna \cite{Arias}, and Conrey \cite{Conrey} for level curves for $\zeta(s)$.  Conrey calls the $\zeta(s)$ analogs of the former family G-curves; they cross the critical line at a Gram point.  The analogs of the latter family are called Z-curves; they cross the critical line at a zero of $\zeta(s)$.  We will use Conrey's terminology Z-curves here as well.  (There is no analog of Gram point for $\zeta^\prime(s)$).

There were 19,803 zeros of $\zeta^\prime(s)$ in the data lying on a Z-curve (corresponding to odd $n$ with $1,290,635,525\le n\le1,290,675,131$).  All but one were type 0; all but 17 had real part $\beta^\prime>1$.  We make the following conjecture:
\begin{conjecture}  Asymptotically, 100\% of the zeros lying on a Z-curve are type 0, and have real part $\beta^\prime>1$.
\end{conjecture}

\begin{figure}
\begin{center}
\includegraphics[scale=1.3, viewport=0 0 310 160,clip]{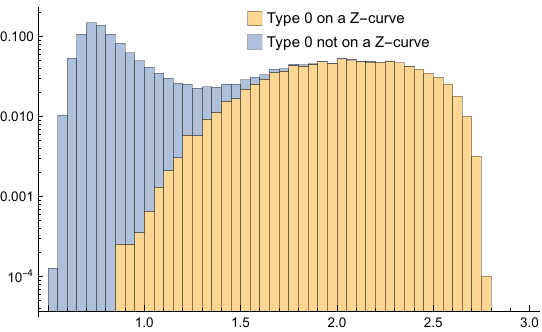}
\caption{Histogram of $\beta^\prime$ (unscaled) for type 0 zeros.}\label{F:3prime}
\end{center}
\end{figure}

Figure \ref{F:3prime} shows a histogram of $\beta^\prime$ (unscaled) for the type 0 zeros lying on a Z-curve (in tan) and those type 0 zeros not on a Z-curve (in blue).  (Since up to height $T$ there are $\sim \log(3/2)/2\pi\cdot  T$ which lie on a Z-curve, and so (conjecturally) $\gg T\log T$ which do not, the vertical axis in Figure \ref{F:3prime} is on a logarithmic scale).    One sees that they seem to have very different distributions.

 \section{Two Lemmas}\label{S:lemmas}

For the Lemmas which will finish the proof of Theorem \ref{Thm:classification}, let $\log\eta(s)$ be any choice of the branch of the logarithm in an open set which contains the critical line but no zeros of $\zeta^\prime$.  
Let
\[
F(t)\overset{\text{def.}}=-\tre\frac{\eta^\prime}{\eta}\left(\frac12+i t\right)=-\tre \frac{d}{d\sigma}\log(\eta(s))_{|s=1/2+it}.
\]
By the Cauchy-Riemann equations applied to $\log(\eta(s))$ on such a set, 
\[
\frac{d\log|\eta(s)|}{d\sigma}=\frac{d\arg\eta(s)}{dt},
\]
and so when evaluated at $s=1/2+it$,
\[
F(t)=-\frac{d \arg(\eta(1/2+it))}{dt}.
\]
\begin{lemma}\label{Lemma:ZhangGe}
For $t>4$, $F(t)>0$.
\end{lemma}
\begin{proof}
In \cite[(2.9)]{Zhang}, Zhang deduces from the Hadamard product for $\eta(s)$ that
\[
F(t)=-\sum_{\lambda^\prime}\tre\left(\frac{1}{1/2+it-\lambda^\prime}\right)+O\left(1\right)
\]
(where $\lambda^\prime$ denotes complex zeros of $\zeta^\prime(s)$ with real part $>1/2$.)\ \ 
In \cite[Lemma 8]{FanGe2}, Ge improves this to get
\begin{equation}\label{Eq:FanGe}
F(t)=-\sum_{\lambda^\prime}\tre\left(\frac{1}{1/2+it-\lambda^\prime}\right)+\log(2)/2+O\left(1/t\right).
\end{equation}
We will not reproduce those calculations (the interested reader can go to the cited work), but observe the $O(1/t)$ error term is explicitly (writing $s=1/2+it$)
\[
\tre\left(\frac{1}{s}-\frac{2}{s-1}\right)-\tre\sum_{n=1}^\infty\left(\frac{1}{s+a_n}-\frac{1}{s+2n}\right).
\]
Here
\[
-(2n+2)<-a_n<-2n
\]
is the unique real zero on $\zeta^\prime(s)$ in the interval.  For $s=1/2+it$,
\[
\tre\left(\frac{1}{s}+\frac{2}{s-1}\right)=-\frac{2}{1+4t^2}.
\]
We claim that the tail of the series,
\[
-\tre\sum_{2n>t}^\infty\left(\frac{1}{s+a_n}-\frac{1}{s+2n}\right)>0,
\]
as this sum is
\[
\sum_{2n\ge t}\tre\left(\frac{a_n-2n}{(s+a_n)(s+2n)}\right)=
\sum_{2n\ge t}(a_n-2n)\frac{\tre\left((\bar{s}+a_n)(\bar{s}+2n)\right)}{|s+a_n|^2|s+2n|^2}.
\]
With $2n<a_n<2n+2$, and $2n>t$, every term is positive.
Meanwhile
\[
-\tre\sum_{2n<t}^\infty\left(\frac{1}{s+a_n}-\frac{1}{s+2n}\right)=
\sum_{2n<t}\frac{-a_n-1/2}{|s+a_n|^2}+\frac{1/2+2n}{|s+2n|^2}.
\]
From $|s+a_n|^2>|s+2n|^2$ we deduce
\[
\frac{-1/2-a_n}{|s+a_n|^2}>\frac{-1/2-a_n}{|s+2n|^2},
\]
so this sum is bounded below by
\[
\sum_{2n<t}\frac{-a_n+2n}{|s+2n|^2}>-2\sum_{2n<t}\frac{1}{(2n+1/2)^2+t^2}.
\]
The sum has $t/2$ terms, each less than $1/t^2$ so this sum is bounded below by $-1/t$.  We conclude that
\[
F(t)>-\sum_{\beta^\prime>1/2}\tre\frac{1}{1/2+it-\lambda^\prime}+\log(2)/2+\frac{2}{1+4t^2}-\frac{1}{t},
\]
and for $t>4$,
\[
\log(2)/2-\frac{2}{1+4t^2}-\frac{1}{t}>0.
\]
\end{proof}
By Stirling's formula we have that for $\rho=1/2+i\gamma$ a zero of $\zeta(s)$, Zhang's Lemma 3 is, more explicitly,
\begin{equation}\label{Eq:ZLemma3}
F(\gamma)=\frac12\log\left(\frac{\gamma}{2\pi}\right)+O\left(\frac1\gamma\right).
\end{equation}
Zhang's Lemma 4 becomes
\begin{lemma} For $n\ge 1$,
\[
\int_{\gamma_n}^{\gamma_{n+1}}F(t)\, dt=\pi.
\]
\end{lemma}

We will also need a result analogous to Lemma \ref{Lemma:ZhangGe} on the line $\tre(s)=4$.  Let
\[
f(t)\overset{\text{def.}}=\tre\frac{\eta^\prime}{\eta}\left(4+i t\right).
\]
(Note that unlike $F(t)$, there is no leading minus sign.)\ \ 
\begin{lemma}\label{Lemma:ZhangGe3}
For $t>40$, 
\[
f(t)>0.
\]
\end{lemma}
\begin{proof}
First we consider those $t\ge 3000$.  For $\sigma>1$ we have a Dirichlet series expansion
\[
\zeta^{\prime\prime}(s)=\sum_{n=2}^\infty \log(n)^2n^{-s}.
\]
Along the vertical line $\tre(s)=\sigma$ we have upper bounds
\[
|\zeta^{\prime\prime}(s)|\le \sum_{n=2}^\infty\frac{\log(n)^2}{n^\sigma}.
\]
Comparing sums to integrals we get the upper bound
\begin{multline*}
 |\zeta^{\prime\prime}(\sigma+it)|\le \sum_{n=2}^{20}\frac{\log(n)^2}{n^\sigma}+\\
 \frac{(\sigma -1)^2 \log ^2(20)+2 (\sigma -1) \log (20)+2}{20^{\sigma -1}(\sigma -1)^3}
\end{multline*}

For lower bounds on $|\zeta^\prime(\sigma+it)|$, we write
\begin{align*}
(1-2^{-s})\zeta(s)=&\sum_{n=0}^\infty\frac{1}{(2n+1)^s}\\
(1-2^{-s})\zeta^\prime(s)+2^{-s}\log(2)\zeta(s)=&-\sum_{n=1}^\infty\frac{\log(2n+1)}{(2n+1)^s}.
\end{align*}
Thus
\[
\left|(1-2^{-s})\zeta^\prime(s)\right|\ge 2^{-\sigma}\log(2)\left|\zeta(s)\right|-\sum_{n=1}^\infty\frac{\log(2n+1)}{(2n+1)^\sigma}.
\]
For $\sigma>1$ we have the lower bound bounds
\[
 |\zeta(s)|\ge \frac{\zeta(2\sigma)}{\zeta(\sigma)},
\]
which along with
\[
\sum_{n=21}^\infty\frac{\log(2n+1)}{(2n+1)^\sigma}\le\int_{20}^\infty\frac{\log(2x+1)}{(2x+1)^\sigma}\, dx
=\frac{ (\sigma -1) \log (41)+1}{2\cdot 41^{\sigma -1} (\sigma -1)^2}
\]
gives
\begin{multline*}
\left|(1-2^{-s})\zeta^\prime(s)\right|\ge 2^{-\sigma}\log(2)\frac{\zeta(2\sigma)}{\zeta(\sigma)}\\
- \sum_{n=1}^{20}\frac{\log(2n+1)}{(2n+1)^\sigma}-\frac{ (\sigma -1) \log (21)+1}{2\cdot 21^{\sigma -1} (\sigma -1)^2}.
\end{multline*}
Thus we have the lower bound 
\begin{multline}
|\zeta^\prime(s)|\ge\\
(1+2^\sigma)^{-1}\left(2^{-\sigma}\log(2)\frac{\zeta(2\sigma)}{\zeta(\sigma)}
- \sum_{n=1}^{20}\frac{\log(2n+1)}{(2n+1)^\sigma}-\frac{ (\sigma -1) \log (41)+1}{2\cdot 41^{\sigma -1} (\sigma -1)^2}\right).
\end{multline}
In particular, we estimate $|\zeta^{\prime\prime}(4+it)/\zeta^\prime(4+it)|\le 3.07718$.

Meanwhile $h^\prime/h(4+it)=(\psi(2+it/2)-\log(\pi))/2$, where the polygamma function $\psi(z)=\Gamma^\prime(z)/\Gamma(z)$ has real part which is an increasing function of $t$, and one calculates in \emph{Mathematica} that 
\[
\tre((\psi(2+i3000/2)-\log(\pi))/2)\ge3.084.
\]
In the range $40\le t\le 3000$, we can compute $f(t)$ in \emph{Mathematica} to see it is positive.
\end{proof}
 
For the range $4\le t\le 40$, $\arg(\eta(4+it))$ is \emph{not} increasing, but an examination of Figure \ref{F:2} suffices to complete the proof of Theorem \ref{Thm:classification}.

\section{A closer look at the type 2 zeros.  Curvature of the level curve.}\label{S:Type2}

Since the data indicate that pairs of type 2 zeros of $\zeta(s)$ are closer than average, and type 2 zeros of $\zeta^\prime(s)$ are closer to the critical line, it is natural to ask if one can show
\begin{equation}\label{Eq:varSound}
\liminf (\beta^\prime -1/2)\log\gamma^\prime = 0 \Leftrightarrow \liminf (\gamma^+-\gamma^-)\log \gamma^\prime=0
\end{equation}
when the $\liminf$ are both restricted to the subsequence of type 2 zeros. 

We can investigate \eqref{Eq:varSound} via a study of the curvature of the level curve which connects the triple $1/2+i\gamma^-$, $\rho^\prime$, and $1/2+i\gamma^+$.  With $\eta(s)=u+iv$, the formula for the curvature $\kappa$ of the level curve $u(\sigma,t)=0$ may be found in \cite[\S 3]{Goldman}.  Via the Cauchy-Riemann equations, one sees \cite{Jerrard}
\begin{equation}\label{Eq:kappa1}
\kappa=|\eta^\prime|\cdot \tre(\eta^{\prime\prime}/\eta^{\prime\,2}).
\end{equation}

The osculating circle for the curve $\tre(\eta)=0$ at $\rho^\prime$ has radius $R=1/|\kappa|$.  We will make the assumption that $\rho^\prime$ is the rightmost point on the osculating circle.  Of course this is not true for every type 2 zero, but based on our study of the angles $\theta=\arg(\eta^\prime(\rho^\prime))$ below, it is an approximation which is true in the limit for a sequence of type 2 zeros with $(\beta^\prime-1/2)\log(\gamma^\prime)\to 0$.\ \   Then elementary geometry shows the length of the chord this circle cuts from the critical line is
\[
2\left((\beta^\prime-1/2)(2/|\kappa|+1/2-\beta^\prime)\right)^{1/2}.
\]

The inequality 
\[
2\left((\beta^\prime-1/2)(2/|\kappa|+1/2-\beta^\prime)\right)^{1/2}<2^{3/2}\left(\frac{\beta^\prime-1/2}{|\kappa|}\right)^{1/2}
\]
shows that, along a subsequence with 
\[
(\beta^\prime-1/2)\log(\gamma^\prime)\to 0,
\]
any lower bound of the form
\begin{equation}\label{Eq:needestimate}
\log(\gamma^\prime)\ll |\kappa|
\end{equation}
will force the length of the chord, multiplied by $\log(\gamma^\prime)$, to tend to $0$.  Since the error between the level curve and the osculating circle is of cubic order, we see that (\ref{Eq:needestimate}) will force
\[
(\gamma^+-\gamma^-)\log(\gamma^\prime)\to 0
\]
as well, which will prove the $\Rightarrow$ implication in \eqref{Eq:varSound}.  

The expression (\ref{Eq:kappa1}) for $\kappa$ simplifies when evaluated at a zero $\rho^\prime$ of $\zeta^\prime(s)$.  Recalling the notation $h(s)=\pi^{-s/2}\Gamma(s/2)$, we have
\[
\eta^\prime(\rho^\prime)=h(\rho^\prime)\zeta^{\prime\prime}(\rho^\prime),\qquad \eta^{\prime\prime}(\rho^\prime)=2h^\prime(\rho^\prime)\zeta^{\prime\prime}(\rho^\prime)+h(\rho^\prime)\zeta^{\prime\prime\prime}(\rho^\prime).
\]
Recall $\theta$ denotes $\arg(\eta^\prime(\rho^\prime))$, so
\[
\frac{|\eta^\prime|}{\eta^\prime}(\rho^\prime)=\exp(-i\theta),
\]
and the curvature at $\rho^\prime$ reduces to
\begin{equation}\label{Eq:kappa2}
\kappa=\tre\left(\exp(-i\theta)\left(\frac{2h^\prime}{h}(\rho^\prime)+\frac{\zeta^{\prime\prime\prime}}{\zeta^{\prime\prime}}(\rho^\prime)\right)\right).
\end{equation}

Thus the curvature $\kappa$ of the level curve at $\rho^\prime$ is strongly influenced by the location relative to $\rho^\prime$ of the zeros of $\zeta^{\prime\prime}$, which we investigate in the next two sections.

\section{Zeros of $\zeta^{\prime\prime}(s)$}\label{S:Spira}

Spira, in \cite{Spira} was the first to observe that zeros of successive higher order derivatives of the Riemann zeta function seem to cluster along roughly horizontal lines.  He wrote \lq\lq The zeros of $\zeta^{\prime\prime}$ have imaginary part almost exactly equal to those of $\zeta^\prime$, and lie to the right of them.\rq\rq\ \   (See Figure 1 from his paper)\ \   The following explains Spira's observation.

\begin{theorem}\label{Thm4}
The level curves $\arg(\zeta^{\prime\prime}/\zeta^\prime(s))=0$ connect each zero  of $\zeta^\prime(s)$ with $\tre(s)>1/2$ to a zero of $\zeta^{\prime\prime}(s)$ with $\tre(s)>1/2$ (typically to the right), giving a canonical bijection between these two sets.  The same holds for zeros higher derivatives $\zeta^{(k)}(s)$ and $\zeta^{(k+1)}(s)$.
\end{theorem}
\begin{proof}

Lemma \ref{Lemma:ZhangGe} in \S\ref{S:classify}  already shows that the real part of $\eta^\prime/\eta(s)$ is negative on the critical line $s=1/2+it$ ($t>3$).  To obtain the same result for $\zeta^{\prime\prime}/\zeta^\prime(s)$, we simply subtract off the real part of $h^\prime/h(s)$, for $h(s)=\pi^{-s/2}\Gamma(s/2)$.  By Stirling's Formula this is $\log(t/(2\pi))/2+O(1/t^2)$, so we deduce that the real part of $\zeta^{\prime\prime}/\zeta^\prime(s)$ is negative as well.  Meanwhile, as $\sigma\to +\infty$, $\zeta^{\prime\prime}/\zeta^\prime(s)\to -\log(2)$ and so the real part of $\zeta^{\prime\prime}/\zeta^\prime(s)$ will again be negative.

Now fix a zero $\rho^\prime$ of $\zeta^\prime(s)$, and consider the level curves with 
\linebreak
$\arg\left(\zeta^{\prime\prime}/\zeta^\prime(s)\right)=0$ exiting the pole at $\rho^\prime$.   By the above observations, this contour can't cross the critical line, nor extend too far into the right half plane.  The only possibility is that it terminates.  To finish the argument we must be sure that a contour 
$\arg(\zeta^{\prime\prime}/\zeta^\prime(s))=0$ can not connect a zero of $\zeta^{\prime\prime}(s)$ to another zero nor a pole (i.e. zero of $\zeta^\prime(s)$) to another pole.  But this contour is the inverse image under $\zeta^\pp/\zeta^\prime$ of the positive real axis, connecting $0$ to $\infty$ on the Riemann sphere.  It can only connect zeros to poles.

For $s$ near a zero $\rho^\prime$ of $\zeta^\prime$, elementary manipulations of series expansions give that
$$
\frac{\zeta^{\prime\prime}}{\zeta^\prime}(s)=\frac{1}{s-\rho^\prime}+O(1),
$$
so the contour $\arg(\zeta^{\prime\prime}/\zeta^\prime(s))=0$ has to exit the pole to the right, and so the zeros of $\zeta^{\prime\prime}$ will typically be to the right of the zeros of $\zeta^\prime$.  Similarly, a contour with $\arg(\zeta^{\prime\prime}/\zeta^\prime(s))=0$ terminating in a zero of $\zeta^{\prime\prime}(s)$, when followed backwards, must originate in a pole, i.e., a zero of $\zeta^\prime(s)$.
\end{proof}

This same argument works for zeros of higher derivatives as well. 



\section{Adaptation of a Theorem of Marden}\label{S:Marden}

This section is inspired by the results in \cite{Marden1, Marden2}, which express the logarithmic derivative of an entire function $f$  as a sum over poles (zeros of $f$) weighted by rational expressions in a fixed set of zeros of $f$ and $f^\prime$.   Marden's proof of \cite[Theorem 2.1]{Marden1} via the Cauchy Integral Formula can be generalized to $f(s)=\zeta^\prime(s)$, but in fact  \cite[Theorem 2.1]{Marden1}  and \cite[Theorem 2.1]{Marden2} can be proved simply by a partial fraction decomposition and taking linear combinations of the  Hadamard logarithmic derivative, as shall see.

We denote the real zeros of $\zeta^\prime(s)$ as $-a_n$, with $a_n\in (2n,2n+2)$.  In fact,
\[
-a_n=-2n-2+\frac{1}{\log(n)}+O\left(\frac{1}{\log^2(n)}\right).
\]

The following proposition will be needed to investigate the curvature $\kappa$ of the level curve at a fixed zero $\rho^\prime$ of $\zeta^\prime(s)$.
A \emph{generic} non-real zero of $\zeta^\prime(s)$ will be denoted as $\lambda^\prime$ in this subsection.  

\begin{proposition}\label{P:Marden}
Fix a complex zero $\rho^\pp=\beta^\pp+i\gamma^\pp$ of $\zeta^\pp(s)$ with $\beta^\pp>1/2$.  With $s=\sigma+it$ in a vertical strip $a\le\sigma\le b$ we have
\begin{multline}\label{Eq:Marden}
\frac{\zeta^\pp}{\zeta^\prime}(s)=\frac{\log(\gamma^\pp/t)}{2}
+
\sum_{\lambda^\prime}\frac{\rho^\pp-s}{(s-\lambda^\prime)(\rho^\pp-\lambda^\prime)}\\
+O\left(\frac{1}{t}\right)+O\left(\frac{1}{\gamma^\pp}\right).
\end{multline}
The sum is uniformly convergent on compact sets.
\end{proposition}
\begin{proof}
The starting point is the partial fractions representation
\begin{multline}\label{Eq:new2}
\frac{\zeta^\pp}{\zeta^\prime}(s)=\frac{\zeta^\pp}{\zeta^\prime}(0)-2-\frac{2}{s-1}+\\
\sum_n\left(\frac{1}{s+a_n}-\frac{1}{a_n}\right)+\sum_{\lambda^\prime}\left(\frac{1}{s-\lambda^\prime}+\frac{1}{\lambda^\prime}\right)
\end{multline}
which follows from the Hadamard theory.   From \eqref{Eq:new2} subtract $0=\zeta^\pp/\zeta^\prime(\rho^\pp)$ to obtain
\begin{multline*}
\frac{\zeta^\pp}{\zeta^\prime}(s)=-\frac{2}{s-1}+\frac{2}{\rho^\pp-1}\\
+\sum_n\left(\frac{1}{s+a_n}-\frac{1}{\rho^\pp+a_n}\right)+\sum_{\lambda^\prime}\frac{\rho^\pp-s}{(s-\lambda^\prime)(\rho^\pp-\lambda^\prime)}.
\end{multline*}
Add and subtract $\psi(s/2)/2$ where the digamma function $\psi(s)=\Gamma^\prime/\Gamma(s)$.  From the series representation for the digamma function we see that
\begin{multline*}
\frac{\zeta^\pp}{\zeta^\prime}(s)=-\frac{1}{2}\psi(s/2)-\frac{C}{2}-\frac{1}{s}-\frac{2}{s-1}+\frac{2}{\rho^\pp-1}\\
+\sum_n\left(\frac{1}{2n}-\frac{1}{s+2n}+\frac{1}{s+a_n}-\frac{1}{\rho^\pp+a_n}\right)+\sum_{\lambda^\prime}\frac{\rho^\pp-s}{(s-\lambda^\prime)(\rho^\pp-\lambda^\prime)}.
\end{multline*}
(Here $C$ is the Euler constant.)  We regroup the terms to obtain
\begin{multline*}
\frac{\zeta^\pp}{\zeta^\prime}(s)=-\frac{1}{2}\psi(s/2)-\frac{C}{2}-\frac{1}{s}-\frac{2}{s-1}+\frac{2}{\rho^\pp-1}+\\
\sum_n\frac{2n-a_n}{(s+2n)(s+a_n)}+
\sum_n\left(\frac{1}{2n}-\frac{1}{\rho^\pp+a_n}\right)+
\sum_{\lambda^\prime}\frac{\rho^\pp-s}{(s-\lambda^\prime)(\rho^\pp-\lambda^\prime)}.
\end{multline*}
\begin{lemma}
We have
\[
-\frac{1}{s}-\frac{2}{s-1}+\sum_n\frac{2n-a_n}{(s+2n)(s+a_n)}=O\left(\frac{1}{t}\right)
\]
\end{lemma}
\begin{proof}
The numerator $2n-a_n$ of the summand is $O(1)$, while
\begin{multline*}
\sum_n\frac{1}{(s+2n)(s+a_n)}\ll |s|^2\sum_n\frac{1}{(4n^2+t^2)^2}+\\
|s|\sum_n\frac{4n}{(4n^2+t^2)^2}+\sum_n\frac{4n^2}{(4n^2+t^2)^2}.
\end{multline*}
The first and last sum on the right have complicated closed forms in terms of $\sinh$, $\cosh$, $\coth$, and $\csch$, while the middle sum is expressed in terms of $\psi^\prime$, the derivative of the digamma function.  Including the leading factors $|s|^2$, $|s|$ and $1$, each is $O(1/t)$.
\end{proof}
We add
\[
0=\frac{1}{2}\psi(\rho^\pp/2)+\frac{C}{2}+\frac{1}{\rho^\pp}+\sum_n\left(\frac{1}{\rho^\pp+2n}-\frac{1}{2n}\right)
\]
to see that
\begin{multline*}
\frac{\zeta^\pp}{\zeta^\prime}(s)=\frac{\psi(\rho^\pp/2)-\psi(s/2)}{2}
+\sum_{\lambda^\prime}\frac{\rho^\pp-s}{(s-\lambda^\prime)(\rho^\pp-\lambda^\prime)}
\\
+\frac{1}{\rho^\pp}+\frac{2}{\rho^\pp-1}+\sum_n\left(\frac{1}{\rho^\pp+2n}-\frac{1}{\rho^\pp+a_n}\right)+O\left(\frac{1}{t}\right).
\end{multline*}
Via the lemma,
\[
\frac{1}{\rho^\pp}+\frac{2}{\rho^\pp-1}
+
\sum_n\left(\frac{1}{\rho^\pp+2n}-\frac{1}{\rho^\pp+a_n}\right)=O\left(\frac{1}{\gamma^\pp}\right).
\]
Stirling's formula gives
\[
\frac{\psi(\rho^\pp/2)-\psi(s/2)}{2}=\frac{1}{2}\log(\gamma^\pp/t)+O\left(\frac{1}{t}\right)+O\left(\frac{1}{\gamma^\pp}\right).
\]
\end{proof}

\begin{figure}
\begin{center}
\includegraphics[scale=1.1, viewport=0 0 150 500,clip]{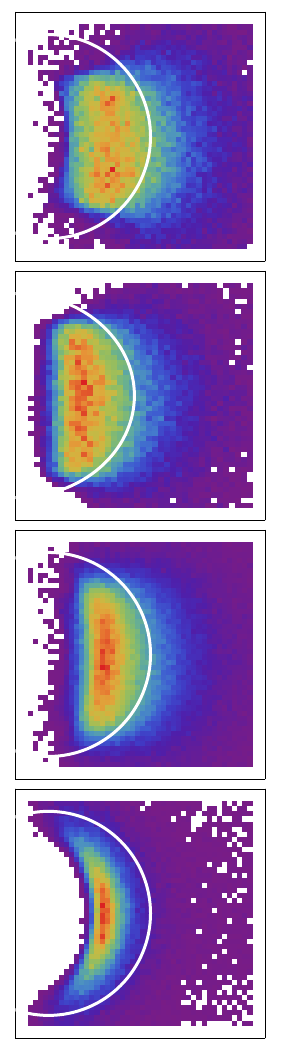}
\caption{Position of $\rho^{\prime\prime}$ relative to $\rho^\prime$, \\\hspace{\textwidth}
shown by quartiles of $(\beta^\prime-1/2)\log \gamma^\prime$}\label{F:10}
\end{center}
\end{figure}

As a consequence of  Proposition \ref{P:Marden} and Stirling's formula, we note that for
$F(t)=-\tre\left(\eta^\prime/\eta\left(\frac12+i t\right)\right)$ as in Lemma \ref{Lemma:ZhangGe}, we have
\begin{multline}\label{Eq:newF(t)}
F(t)=-\frac{1}{2}\log(\gamma^\pp/2\pi)-
\sum_{\lambda^\prime}\tre\left(\frac{\rho^\pp-s}{(s-\lambda^\prime)(\rho^\pp-\lambda^\prime)}\right)\\
+O\left(\frac{1}{t}\right)+O\left(\frac{1}{\gamma^\pp}\right).
\end{multline}

\begin{theorem}\label{Thm:fund}  Let $\rho^\prime$ denote the zero of $\zeta^\prime(s)$ canonically associated  via Theorem \ref{Thm4} to $\rho^\pp$.  Subtracting \eqref{Eq:FanGe}  from \eqref{Eq:newF(t)} gives
\begin{equation}\label{Eq:fund}
\tre\left(\frac{1}{\rho^\pp-\rho^\prime}\right)+\sum_{\lambda^\prime\ne\rho^\prime}\tre\left(\frac{1}{\rho^\pp-\lambda^\prime}\right)=\frac{\log(\gamma^\pp/\pi)}{2}+O\left(\frac{1}{\gamma^\pp}\right).
\end{equation}
\end{theorem}
(The $O(1/t)$ term drops out as the rest of the expression is independent of $t$.)

Theorem \ref{Thm:fund} is a fundamental identity that relates the location of the $\rho^\pp$ associated to $\rho^\prime$ to the location  of all the  $\lambda^\prime\ne\rho^\prime$.  It gives a heuristic explanation of an intriguing new phenomena we see in the data.  Just as small gaps between the Riemann zeros tend to be associated with $\rho^\prime$ close to the critical line, $\rho^\prime$ close to the critical line has an effect on the position of the canonically associated zero $\rho^{\prime\prime}$ of $\zeta^{\prime\prime}(s)$ as well:   Figure \ref{F:10}  shows the 273,422 type 2 zeros sorted into quartiles of $(\beta^\prime-1/2)\log(\gamma^\prime)$.  So at the top of Figure \ref{F:10} we have data for the quartile furthest from the critical line: $(\beta^\prime-1/2)\log(\gamma^\prime)>1.443$.   At the bottom is the data for the quartile closest to the critical line: $(\beta^\prime-1/2)\log(\gamma^\prime)\le 0.487$.

For each quartile, we show a density histogram of the position of the canonically associated $\rho^{\pp}$ relative to $\rho^\prime+1/\log(\gamma^\prime)$, scaled by $\log(\gamma^\prime)$.  Red denotes the most points in a bin, purple the fewest.  With this normalization the circle is the unit circle, shown in white.  One sees more or less random behavior for the  quartile farthest from the critical line (top).  As we go to  quartiles  for smaller values of $(\beta^\prime-1/2)\log(\gamma^\prime)$, the zeros $\rho^{\pp}$ seem to be both less likely to be near $\rho^\prime$, and more likely to be near the circle.  

Heuristically, if $\rho^\prime$ is close to the critical line, with  $\rho^\pp$ lying to the right of $\rho^\prime$ we expect there to be $\lambda^\prime$ with $\tre(1/(\rho^\pp-\lambda^\prime))$ positive as well as negative, so there is cancellation in the sum \eqref{Eq:fund}.  So we expect that when $\rho^\prime$ is close to the critical line,
\[
\tre\left(\frac{1}{\rho^\pp-\rho^\prime}\right)\approx\frac{\log(\gamma^\pp/\pi)}{2}\approx\frac{\log(\gamma^\prime)}{2}.
\]
Since the level curves $\tre(1/(z-\rho^\prime))=c/2$ are circles with center $\rho^\prime+1/c$ and  radius $1/c$, we expect that  $\rho^\pp$ lies near a circle of radius  approximately $1/\log(\gamma^\prime)$ and center $\rho^\prime+1/\log\gamma^\prime$.\label{page12}

\section{Curvature at $\rho^\prime$ in terms of all $\lambda^\prime\ne\rho^\prime$}\label{S:application}

To study the curvature, we first need two auxiliary results on the location of $\gamma^\prime$ relative to $\gamma^+$, $\gamma^-$, and on $\theta=\arg(\eta^\prime(\rho^\prime))$.

We claim that $\gamma^\prime$ is very near to $t_0=(\gamma^++\gamma^-)/2$ when either $\beta^\prime-1/2$ is small or $\gamma^+-\gamma^-$ is  small.  In fact, borrowing the notation of  \cite[p.50-51]{Stopple.Lehmer} we introduce $\Delta, t_0$, $Y$, and $\lambda$ defined by\label{page13}
\begin{gather*}
\Delta=\gamma^+-\gamma^-,\qquad \lambda=\log(t_0/2\pi),\\
\rho^{\pm}=1/2+i(t_0\pm \Delta/2),\qquad
\rho^\prime=\beta^\prime+i(t_0+Y)
\end{gather*}
so $\gamma^\prime=t_0+Y$.
We  rescale with 
\[
x=(\beta^\prime-1/2)\lambda,\qquad y=Y\lambda,\qquad\delta=\Delta\lambda/2\pi.
\]
In \cite[p.50-51]{Stopple.Lehmer} we developed series expansions
\begin{gather*}
x(\delta)=\frac{\pi^2}{4}\left(1-\frac{\log(\pi)}{\lambda}\right)\delta^2+O(\delta^4),\\
y(\delta)=\frac{\pi^2}{2\lambda}\left(\frac{\pi}{4}+\sum_{\rho\ne\rho^\pm}\frac{1}{t_0-\gamma}\right)\,\, \delta^2+O(\delta^4).
\end{gather*}
In the  first, we estimate $\delta^2$ in terms of $x$ and plug into the second.  We neglect the sum over $\rho\ne\rho^\pm$, which should show significant cancellation.  (For an individual example of a $\rho^\prime$, there may be an imbalance with more nearby $\rho$ above $\rho^+$ than below $\rho^-$ or vice versa.  But as we will be considering an infinite sequence of type 2 pairs, the only  way the result could fail is if all but finitely many of the pairs showed such an imbalance, which is not plausible.)\ \   Converting back to the original variables we get
\begin{equation}\label{Eq:last}
\gamma^\prime=t_0+O\left(\gamma^+-\gamma^-\right)^2,\,\,\text{ and }\,\,   \gamma^\prime = t_0+O\left(\frac{\beta^\prime-1/2}{\log(\gamma^\prime)}\right).
\end{equation}

We will next investigate the angles $\theta$, by comparing to the argument of $\eta(1/2+i\gamma^\prime)$.  We observe that the argument of $\eta(1/2+it)$ changes from $-\pi/2$ to $\pi/2 \bmod 2\pi$ (or the reverse) as $t$ increases from $\gamma^-$ to $\gamma^+$, so the argument of $\eta(1/2+it_0)$ will be very near to either $0$ or $\pi\bmod 2\pi$, and so will the argument of $\eta(1/2+i\gamma^\prime)$.  Since $\eta(\rho^\prime)=0$, the argument at $\rho^\prime$ is not defined, but there is a limiting value along the horizontal line $\sigma+i\gamma^\prime$ as $\sigma$ approaches $\beta^\prime$ from the left.  Consider the Taylor expansion of $\eta$ at $\rho^\prime$:
\[
\eta(\sigma+i\gamma^\prime)=h\zeta^{\prime\prime}(\rho^\prime)(\sigma-\beta^\prime)+O\left(\sigma-\beta^\prime\right)^2,
\]
so
\[
\frac{\eta(\sigma+i\gamma^\prime)}{\sigma-\beta^\prime}=h\zeta^{\prime\prime}(\rho^\prime)+O\left(\sigma-\beta^\prime\right).
\]
(We caution that this expression is for a \emph{fixed} $\rho^\prime$, and the big $O$ expression holds for $\sigma-\beta^\prime\to 0$.)
Because $\sigma-\beta^\prime$ is negative and real for $1/2<\sigma<\beta^\prime$, we see that
\begin{equation}\label{Eq:thetaest}
\lim_{\sigma\to\beta^\prime}\arg\left(\eta(\sigma+i\gamma^\prime)\right)\equiv \arg\left(h\zeta^{\prime\prime}(\rho^\prime)\right)+\pi\bmod 2\pi.
\end{equation}
Note the shift by $\pi$ modulo $2\pi$: when the argument of $\eta(1/2+i\gamma^\prime)$ is very near to $0$ (resp.\ $\pi$),  \eqref{Eq:thetaest} implies $\theta=\arg(h\zeta^{\prime\prime}(\rho^\prime))$ is near $\pi$ (resp.\ $0$) modulo $2\pi$.   
This in turn implies that $\exp(-i\theta)$ is  near $-1$ (resp.\ $1$), and does not contribute significantly to the formula \eqref{Eq:kappa2} for $|\kappa|$. 

The next proposition relates the remaining parameters determining the curvature of the level curve at one zero $\rho^\prime$ of $\zeta^\prime(s)$ to the locations of all the other zeros $\lambda^\prime\ne\rho^\prime$ of $\zeta^\prime(s)$.

\begin{proposition}   Let $\rho^\prime$ be a zero of $\zeta^\prime(s)$.  Then
\begin{multline}\label{Eq:Prop13}
\tre\left(2\frac{h^\prime}{h}(\rho^\prime)+\frac{\zeta^{\prime\prime\prime}}{\zeta^{\pp}}(\rho^\prime)\right)=\\
-\log(2)-2\sum_{\lambda^\prime\ne\rho^\prime}\tre\left(\frac{1}{\lambda^\prime-(1/2+i\gamma^\prime)}\right)\\
+O\left(\beta^\prime-1/2\right)+O\left(\frac{1}{\gamma^\prime}\right).
\end{multline}
\end{proposition}
\begin{proof}
With $\rho^\prime=\beta^\prime+i\gamma^\prime$, let $\rho^{\prime\prime}=\beta^{\prime\prime}+i\gamma^{\prime\prime}$ be the zero of $\zeta^{\prime\prime}(s)$ canonically associated via Theorem \ref{Thm4}.
We evaluate $\zeta^\pp/\zeta^\prime(s)$ at $s=1/2+i\gamma^\prime$ using Proposition \ref{P:Marden}, and also via a Laurent expansion at $\rho^\prime$.  In fact the algebra of power series gives
\[
\frac{\zeta^\pp}{\zeta^\prime}(1/2+i\gamma^\prime)=\frac{1}{1/2-\beta^\prime}+\frac{\zeta^{\prime\prime\prime}}{2\zeta^{\pp}}(\rho^\prime)+O\left(\beta^\prime-1/2\right).
\]

Stirling's formula applied to $2h^\prime/h(\rho^\prime)$ and taking real parts shows that
\begin{multline}
\tre\left(2\frac{h^\prime}{h}(\rho^\prime)+\frac{\zeta^{\prime\prime\prime}}{\zeta^{\pp}}(\rho^\prime)\right)=
\log(\gamma^\pp/\pi)-\log(2)
+\\
2\sum_{\lambda^\prime\ne\rho^\prime}\tre\left(\frac{\rho^\pp-(1/2+i\gamma^\prime)}{((1/2+i\gamma^\prime)-\lambda^\prime)(\rho^\pp-\lambda^\prime)}\right)-\tre\left(\frac{2}{\rho^\pp-\rho^\prime}\right)\\
+O\left(\beta^\prime-1/2\right)+O\left(\frac{1}{\gamma^\prime}\right)+O\left(\frac{1}{\gamma^\pp}\right).
\end{multline}
We then use \eqref{Eq:fund} to replace
\[
\log(\gamma^\pp/\pi)-\tre\left(\frac{2}{\rho^\pp-\rho^\prime}\right)
\quad
\text { with  }
\quad
2\sum_{\lambda^\prime\ne\rho^\prime}\tre\left(\frac{1}{\rho^\pp-\lambda^\prime}\right).
\]
Finally, $1/\gamma^{\prime\prime}$ is certainly $O(1/\gamma^\prime)$.
\end{proof}

We now summarize what needs to be done to prove \eqref{Eq:varSound}.  For the forward implication, by (\ref{Eq:needestimate}), (\ref{Eq:kappa2}), and (\ref{Eq:Prop13}), it suffices to show:
\begin{sconjecture}  For  a sequence of type 2 zeros $\rho^\prime$ with 
\[(\beta^\prime-1/2)\log(\gamma^\prime)\to 0,
\]
we have
\begin{equation}\label{Eq:sconjecture}
\log(\gamma^\prime)\ll \sum_{\lambda^\prime\ne\rho^\prime}\tre\left(\frac{1}{\lambda^\prime-(1/2+i\gamma^\prime)}\right).
\end{equation}
\end{sconjecture}

We will show this in part II.  (To be clear, the $O(\beta^\prime-1/2)$ and $O(1/\gamma^\prime)$ terms in (\ref{Eq:Prop13}) are for a \emph{fixed} zero $\rho^\prime$.  In order that \eqref{Eq:sconjecture} imply \eqref{Eq:varSound}, one will need to understand how the implied constants vary with $\rho^\prime$ in order to neglect these terms.)

For the reverse implication, consider a sequence of type 2 pairs $\rho^+$, $\rho^-$, with 
\[
(\gamma^+-\gamma^-)\log((\gamma^++\gamma^-)/2))\to 0.
\]
We know from \cite[Theorem  2]{FanGe2}, that 
for any
$v<0.4$ the following holds: For all sufficiently large $\gamma^+$, $\gamma^-$ with $\Delta< v/\log t_0$, (notation as in page \pageref{page13}) the box
$$
\{s=\sigma+it: \frac{1}{2}<\sigma<\frac{1}{2}+\frac{v^2}{4\log t_0},
\gamma^-\le t\le \gamma^+\}
$$
contains exactly one zero $\rho^\prime$  of $\zeta^\prime(s)$.  To show the desired implication we just need to confirm for this type 2 pair $\gamma^-$, $\gamma^+$, that  $\rho^\prime$ as above is the canonically associated type 2 zero, and not some stray type 1 or  type 0.  This will also be done in part II.

 \section{Appendix: Random Matrix Analogs}\label{S:Appendix}
 
There's a productive analogy between the zeros of the Riemann zeta function $\zeta(s)$ and the zeros of the characteristic polynomial $p_A(z)$ of a unitary matrix $A$, and even the zeros of the respective derivatives.  In this analogy the critical line $\tre(s)=1/2$  corresponds to the unit circle $|z|=1$.  More precisely, the analogy relates zeros of $\zeta(s)$ at height $T$ in the critical strip  and zeros of $p_A(z)$ when $A$ is $n$ by $n$, for $n\approx \log(T)$.  On the number theory side, the function
\[
\eta(s)=\pi^{-s/2}\Gamma(s/2)\zeta^\prime(s)
\]
is very useful.   Here we develop an analog $\eta_A(z)$ for the characteristic polynomial of a unitary matrix $A$, and explore its applications.

In \S \ref{AS1}, Theorem \ref{AThm1} says that if
$\exp(i\theta^+)$ and $\exp(i\theta^-)$ are \lq consecutive\rq\ zeros of $p_A(z)$ with 
\[
\theta^+-\theta^-<\frac{2\pi}{1+6n},
\]
then the angular sector
\[
S(\theta^+,\theta^-)=\left\{ 1-2(\theta^+-\theta^-)< |z|,\,  \theta^-\le\arg(z)\le \theta^+  \right\}
\]
contains a zero of $p_A^\prime(z)$.

Some of the results in \S \ref{AS1} are also in \cite{GeGonek}; We think the introduction of the function $\eta_A(z)$ clarifies the analogy to the number theory side.

In \S \ref{AS2}, Theorem \ref{AThm2}  uses the level curves $\tre(\eta_A(z))=0$ to classify the zeros of $p_A(z)$ and $p_A^\prime(z)$, analogous to the classification of the zeros of $\zeta(s)$ and $\zeta^\prime(s)$ above.

\subsection{Results about location of zeros $p_A(z)$ and $p_A^\prime(z)$}\label{AS1}
Let $A$ be an $n\times n$ unitary matrix, $n>2$, with distinct eigenvalues and characteristic polynomial $p_A(z)=\det(A-z I)$.  

We define
\[
h_A(z)=(-z)^{-n/2}\det(A)^{-1/2},\qquad \Lambda_A(z)=h_A(z) p_A(z).
\]
If $n$ is even, $\Lambda_A$ has a pole of order $n/2$ at the origin.  If $n$ is odd, there is a branch cut from the origin to infinity.  (The location of the branch cut does not particularly matter, but when
we denote $\exp(i\theta^-)$, $\exp(i\theta^+)$ as consecutive zeros of $p_A(z)$, we will assume it does not pass between them.)\ \ 
Since
\begin{multline*}
(-z)^np_A(1/z)=\det(-z\cdot A+I)\\
=\det(A)\cdot\det(-z\cdot I+A^*)=\det(A)p_{A^*}(z).
\end{multline*}
we get that
\begin{equation}\label{AEq1}
\Lambda_A(z)=\Lambda_{A^*}(1/z).
\end{equation}
The function $h_A(z)$ is the matrix theory analog of the function $h(s)=\pi^{-s/2}\Gamma(s/2)$ which completes the Riemann zeta function to obtain the functional equation.  We define
\[
\eta_A(z)=z h_A(z)p_A^\prime(z);
\]
this will be the matrix theory analog of the function $\eta(s)=h(s)\zeta^\prime(s)$. (The need for the extra factor of $z$ will become apparent.)\ \ With this we can easily prove for characteristic polynomials of unitary matrices some interesting analogs of results for $\zeta(s)$ and $\zeta^\prime(s)$.

\begin{lemma}\label{ALemma1} The function $z\Lambda_A^\prime(z)$ is purely imaginary on $|z|=1$, while $z\Lambda_A^\prime(z)+z^2\Lambda_A^{\prime\prime}(z)$ is real on $|z|=1$.  \end{lemma}
\begin{proof}
On the unit circle $|z|=1$
\[
\Lambda_A(z)=\Lambda_{A^*}(\overline{z})=\overline{\Lambda_A(z)}
\]
is real valued by  \eqref{AEq1}.  Write $z=\exp(i \theta)$, so
\[
\frac{d}{d\theta}\Lambda_A(\exp(i \theta))=  i \exp(i\theta)\Lambda^\prime(\exp(i\theta))=i z \Lambda_A^\prime(z).
\]
The imaginary part of this function is $0$ on the unit circle.  Similarly, another derivative with respect to $\theta$ gives the second claim.
\end{proof}
The above is the  analog of \cite[(2.1)]{Zhang}, while the following is the analog of \cite[Lemma 1]{Zhang}.
\begin{lemma}
On the unit circle $|z|=1$, we have
\[
p_A(z)=0 \Leftrightarrow \tre(\eta_A(z))=0.
\]
\end{lemma}
\begin{proof}
\begin{multline*}
z\Lambda^\prime(z)=z\left(h_A^\prime(z) p_A(z)+h_A(z)p_A^\prime(z)\right)\\
=z\left(\frac{h_A^\prime(z)}{h_A(z)}\Lambda_A(z)+h_A(z)p_A^\prime(z)\right)\\
=(-n/2)\Lambda_A(z)+zh_A(z)p_A^\prime(z).
\end{multline*}
By   Lemma \ref{ALemma1}, on the unit circle the real part of the left side of the equation is $0$, while $\Lambda_A(z)$ is already real valued.  So $\Lambda_A(z)$ is $0$ if and only if $\eta_A(z)$ is purely imaginary.
\end{proof}

The Lemma makes the level curves $\tre(\eta_A(z))=0$ of interest.  Figure \ref{AF:2} shows an example with $n=22$, using \emph{Mathematica} to generate a random matrix from the Circular Unitary Ensemble.
We use color to indicate the sign of $\tim(\eta_A(z))$.  
Green indicates $\tre(\eta_A(z))=0$ and $\tim(\eta_A(z)))$ is positive, while purple indicates $\tre(\eta_A(z))=0$ and 
$\tim(\eta_A(z))$ is negative.  The  zeros $\mu$ of $p_A(z)$ occur where the green or purple contours cross the unit circle.  The zeros $\mu^\prime$ of $p_A^\prime$ are visible exactly where the two colors come together.

\begin{figure}
\begin{center}
\includegraphics[scale=1, viewport=0 0 260 260,clip]{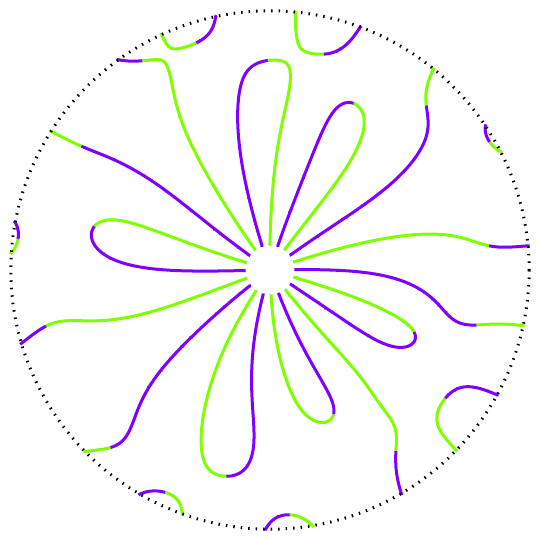}
\caption{An example with $n=22$: Level curves for $\tre(\eta_A(z))=0$}\label{AF:2}
\end{center}
\end{figure}

We define
\[
F(\theta)=\frac{d}{d\theta}\arg(\eta_A(\exp(i\theta))).
\]
This is the matrix analog of the function $F(t)$ in \cite[(2.13)]{Zhang}, and appears as (12) in \cite{GeGonek}.  In particular
\begin{multline*}
\frac{d}{d\theta}\tim(\log(\eta_A(\exp(i\theta))))=\tim\left(i\exp(i\theta)\frac{\eta_A^\prime}{\eta_A}(\exp(i\theta))\right)\\
=\tre\left(\exp(i\theta)\frac{\eta_A^\prime}{\eta_A}(\exp(i\theta))\right).
\end{multline*}

\begin{lemma}
Let $\mu=\exp(i\theta)$ be a zero of $p_A$.  Then $F(\theta)=n/2$.
\end{lemma}
This is Lemma 4.5 in \cite{GeGonek}, and is the matrix analog of \cite[Lemma 3]{Zhang}.
\begin{proof}
By  Lemma \ref{ALemma1}
\begin{multline*}
0=\tre\left(\frac{\mu \Lambda_A^\prime(\mu)+\mu^2\Lambda_A^{\prime\prime}(\mu)}{\mu \Lambda_A^\prime(\mu)}\right)\\
=\tre\left(1+\mu\left(2\frac{h_A^\prime}{h_A}(\mu)+\frac{p_A^{\prime\prime}}{p_A^\prime}(\mu)\right)\right)\\
=\mu\left(\frac{h_A^\prime}{h_A}(\mu)+\frac{\eta_A^\prime}{\eta_A}(\mu)\right).
\end{multline*}
With $h_A^\prime/h_A(\mu)=-n/2\cdot\mu^{-1}$, this says
\[
\frac{n}{2}=\tre\left(\mu\,\frac{\eta_A^\prime}{\eta_A}(\mu)\right).
\]
\end{proof}


\begin{lemma}\label{ALemma4}
Along the unit circle, $\arg(\eta_A(\exp(i\theta)))$ is strictly increasing.  This immediately implies that for consecutive eigenvalues $\exp(i\theta^-)$, $\exp(i\theta^+)$
(in the obvious sense of the counterclockwise orientation of the unit circle), we have
\[
\int_{\theta^-}^{\theta^+}F(\theta)\,d\theta=\pi.
\]
\end{lemma}
This Lemma is the matrix analog of \cite[Lemma 4]{Zhang}.  
\begin{proof}
We have
\begin{align*}
\exp(i\theta)\frac{\eta_A^\prime}{\eta_A}(\exp(i\theta))=&1-n/2+\exp(i\theta)\frac{p_A^{\prime\prime}}{p_A^\prime}(\exp(i\theta))\\
=&1-n/2+\sum_{k=1}^{n-1}\frac{1}{1-\exp(-i\theta)\mu_j^\prime},
\end{align*}
where $\mu_k^\prime$ denotes the zeros of $p_A^\prime(z)$, $k=1,\ldots n-1$.  Regrouping the $1-n/2$, we consider
\[
1/2+\sum_{j=1}^{n-1}\tre\left(\frac{1}{1-\exp(-i\theta)\mu^\prime_j}-1/2\right).
\]
A little algebra shows that
\[
\tre\left(\frac{1}{1-\exp(-i\theta)\mu^\prime_j}-1/2\right)\ge 0\Leftrightarrow \left|\exp(-i\theta)\mu_j^\prime\right|^2\le 1,
\]
which we know to be the case by the Gauss-Lucas theorem, which states that the zeros of $p_A^\prime(z)$ lie in the convex hull of the zeros of $p_A(z)$.  
\end{proof}
In fact we proved that $F(\theta)\ge 1/2$, which gives us the trivial bound on the gaps between zeros.  However the Lemma is necessary for   Theorems \ref{AThm1} and \ref{AThm2}.

\begin{theorem}\label{AThm1}
Let $\mu^+=\exp(i\theta^+)$, $\mu^-=\exp(i\theta^-)$ be zeros of $p_A(z)$ with 
\[
\theta^+-\theta^-<\frac{2\pi}{1+6n}.
\]
Then the angular sector
\[
S(\theta^+,\theta^-)=\left\{ 1-|z| <2(\theta^+-\theta^-),\,  \theta^-\le\arg(z)\le \theta^+  \right\}
\]
contains a zero of $p_A^\prime(z)$.
\end{theorem}
This is the matrix analog of \cite[Lemma 11]{FanGe2}.
\begin{proof}
Via some more algebra, note that
\[
\tre\left(\frac{1}{1-\exp(-i\theta)\mu^\prime}-1/2\right)=\frac{1-|\mu^\prime|^2}{2|1-\exp(-i\theta)\mu^\prime|^2}.
\]
With $d=2(\theta^+-\theta^-)$, write $F(\theta)=1/2+F_{11}(\theta)+F_{12}(\theta)$, where
\begin{align*}
F_{11}(\theta)=&\sum_{1-d<|\mu^\prime|}\frac{1-|\mu^\prime|^2}{2|1-\exp(-i\theta)\mu^\prime|^2}\\
F_{12}(\theta)=&\sum_{|\mu^\prime|\le 1-d}\frac{1-|\mu^\prime|^2}{2|1-\exp(-i\theta)\mu^\prime|^2}.
\end{align*}
(The notation here and below is meant to highlight the analogy with the proof of \cite[Lemma 11]{FanGe2}.)\ \ 
For $|\mu^\prime|\le 1-d$ and $\theta^-<\theta<\theta^+$, we have that
\begin{multline*}
\left|\exp(i\theta)-\mu^\prime\right|\ge \left|\exp(i\theta^+)-\mu^\prime\right|-\left|\exp(i\theta)-\exp(i\theta^+)\right|\\
\ge  \frac{1}{2}\left|\exp(i\theta^+)-\mu^\prime\right| +(\theta^+-\theta^-)-\left|\exp(i\theta)-\exp(i\theta^+)\right|\\
\ge \frac{1}{2}\left|\exp(i\theta^+)-\mu^\prime\right|.
\end{multline*}
Thus 
\[
F_{12}(\theta)\le 4F_{12}(\theta^+)<4F(\theta^+)=2n.
\]
Now suppose there is no zero of $p_A^\prime(z)$ in the angular sector $S(\theta^+,\theta^-)$.  Let $\delta=(\theta^-+\theta^+)/2$.  Then we may write $F_{11}(\theta)=f(\theta)+g(\theta)$, where 
\begin{align*}
f(\theta)=&\sum_{\substack{1-d<|\mu^\prime|\\\theta^+<\arg(\mu^\prime)<\pi+\delta}} \frac{1-|\mu^\prime|^2}{2|1-\exp(-i\theta)\mu^\prime|^2},\\
g(\theta)=&\sum_{\substack{1-d<|\mu^\prime|\\\delta-\pi<\arg(\mu^\prime)<\theta^-}} \frac{1-|\mu^\prime|^2}{2|1-\exp(-i\theta)\mu^\prime|^2}.
\end{align*}
For $\theta^-\le\theta\le\theta^+$ we have $f(\theta)\le f(\theta^+)$ and $g(\theta)\le g(\theta^-)$.  So
\[
F_{11}(\theta)\le f(\theta^+)+g(\theta^-)\le F(\theta^+)+F(\theta^-)=n.
\]
Now
\[
\pi=\int_{\theta^-}^{\theta^+}F(\theta)\, d\theta\le\int_{\theta^-}^{\theta^+} 1/2+3n\, d\theta=(1/2+3n)(\theta^+-\theta^-),
\]
which is a contradiction.
\end{proof}

We next turn to a classification of the zeros of $p_A(z)$ and $p_A^\prime(z)$ based on the level curves.

\subsection{Classification of zeros of $p_A(z)$ and $p_A^\prime(z)$.}\label{AS2}

\subsubsection*{Zeros of $p_A^\prime(z)$}
\begin{itemize}
\item[Type 0:] We will say a zero $\mu^\prime$ of $p_A^\prime(z)$ is of \textsc{type 0} if \emph{neither} of the two level curves $\tre(\eta_A(z))=0, \tim(\eta_A(z))>0$ and $\tre(\eta_A(z))=0, \tim(\eta_A(z))<0$ exiting $\mu^\prime$ cross the unit circle.
\item[Type 1:] We will say a zero $\mu^\prime$ of $p_A^\prime(z)$ is of \textsc{type 1} if \emph{exactly one} of the two level curves $\tre(\eta_A(z))=0, \tim(\eta_A(z))>0$ and $\tre(\eta_A(z))=0, \tim(\eta_A(z))<0$ exiting $\mu^\prime$ crosses the unit circle.
\item[Type 2:] We will say a zero $\mu^\prime$ of $p_A^\prime(z)$ is of \textsc{type 2} if the level curves $\tre(\eta_A(z))=0, \tim(\eta_A(z))>0$ and $\tre(\eta_A(z))=0, \tim(\eta_A(z))<0$ exiting $\mu^\prime$ \emph{both} cross the unit circle.
\end{itemize}

\subsubsection*{Zeros of $p_A(z)$}
\begin{itemize}
\item[Type 1:] We will say a zero $\mu$ of $p_A(z)$ is of \textsc{type 1} if the level curve $\tre(\eta_A(z))=0$ on which it lies, terminates in a zero $\mu^\prime$ which is of type 1.
\item[Type 2:] We will say a zero $\mu$ of $p_A(z)$ is of \textsc{type 2} if the level curve $\tre(\eta_A(z))=0$ on which it lies, terminates in a zero $\mu^\prime$ which is of type 2.
\end{itemize}

In the example of Figure \ref{AF:2} we see six zeros of $p_A^\prime(z)$ of type 0, eight of type 1, and seven of type 2.
\noindent Let
\[
N_1(A)=\sharp\left\{\text{type 1 zeros of } p_A(z)\right\}.
\]
Let
\[
N_2(A)=
\sharp\left\{\text{pairs of type 2 zeros of } p_A(z)\right\}.
\]
For $j=0,1,2$, let
\[
N_j^\prime(A)=\sharp\left\{\text{zeros }\mu^\prime\text{of }p_A^\prime(z)\text{ of type }j\right\}.
\]

\begin{theorem}\label{AThm2}
Every zero of $p_A(z)$ is of either type 1 or type 2.  Thus we have a canonical mapping from the zeros of $p_A(z)$ to those of $p_A^\prime(z)$, which is two to one on the type 2 zeros, and one to one on the type 1 zeros.  Zeros of $p_A^\prime(z)$ of type 0 are precisely those not in the image of this mapping.  The zeros of $p_A(z)$ type 2 are canonically grouped in pairs.

There is always at least one pair of type 2 zeros of $p_A(z)$, and in fact
\begin{equation}\label{AEq:1}
N_2(A)-N_0^\prime(A)=1,
\end{equation}
while
\begin{equation}\label{AEq:2prime}
N_1(A)+2N_0^\prime(A)=n-2.
\end{equation}
\end{theorem}
\begin{proof}
Regarding the mapping, all this is clear except the first statement, which says the contours which cross the unit circle must terminate in exactly one zero of $p_A^\prime(z)$.  The alternatives we must rule out is continuation of the contour on to the center, or looping back to the unit circle.  

For the first possibility, note that the contour $\arg(\eta_A(z))=\pi/2$ (resp.\ $\arg(\eta_A(z))=-\pi/2$) does not exist in isolation; it is part of a continuum which deform smoothly as the argument is varied.  But the argument of $\eta_A(z)$ is increasing around the unit circle by  Lemma \ref{ALemma4}, while decreasing near the origin because of the $z(-z)^{-n/2}$ term.  The contours can only cross over each other where the argument of $\eta_A(z)$ is undefined, at a zero $\mu^\prime$.

The second possibility is also ruled out by  Lemma \ref{ALemma4}, which says that the argument of $\eta_A(z)$ increases monotonically as one goes around the unit circle.

We have  
\begin{equation}\label{AEq:3}
N_1(A)+2N_2(A)=n.
\end{equation}
Via the mapping, $N_1(A)=N_1^\prime(A)$ and $N_2(A)=N_2^\prime(A)$.
Thus we have
\begin{equation}\label{AEq:4}
N_0^\prime(A)+N_1(A)+N_2(A)=n-1.
\end{equation}
Subtracting (\ref{AEq:4}) from  (\ref{AEq:3}) gives (\ref{AEq:1}).
Subtracting (\ref{AEq:3}) from twice (\ref{AEq:4}) gives (\ref{AEq:2prime}).
\end{proof}

In \emph{Mathematica}, we computed and classified zeros of $p_A^\prime(z)$ for 50,000 matrices $A$ with $n=22$ chosen from the Circular Unitary Ensemble.
We found 22.5\% of zeros were of type 0, 50.3\% of zeros were of type 1, and 27.2\% of zeros were of type 2.

\begin{figure}
\begin{center}
\includegraphics[scale=1.3, viewport=0 0 310 180,clip]{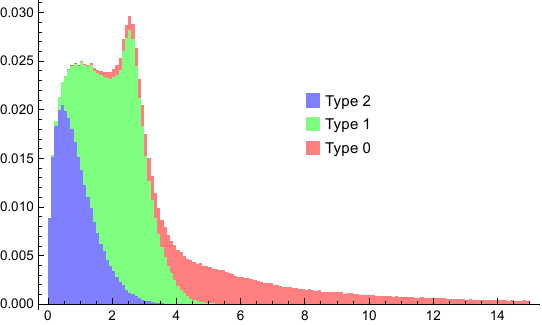}
\caption{$n(1-|\mu^\prime|)$ for 50,000 matrices, $n=22$.}\label{AF:3}
\end{center}
\end{figure}
In Figure \ref{AF:3} we show the histogram of 
$
n\cdot(1-|\mu^\prime|)
$
for zeros of type 0, 1, and 2 separately, for the $\sim 10^6$ zeros $\mu^\prime$ computed.   This is analogous to the data in Figure 2 in \cite{Duenez}, now separated into types, and is the matrix analog of the data in Figure \ref{F:3} above.    Again, the histograms analogous to Figure \ref{AF:3} for the three types separately each show only a single peak; it is the interplay between them that causes the second bump.

The type 2 zeros of $p_A^\prime(z)$ are closer to the unit circle than average: the median value of $n(1-\mu^\prime)$ in the data is $0.78$.  The other two quartiles are $0.43$ and $1.26$.  In contrast, the median for the type 0 zeros is $5.81$.

Corresponding to the type 2 zeros of $p_A^\prime(z)$ in the numerical data, we have the canonically associated pairs of type 2 zeros 
\[
\{\exp(i\theta^-),\exp(i\theta^+)\}
\]
of $p_A(z)$.
The histogram of the normalized gaps 
\[
\frac{n(\theta^+-\theta^-)}{2\pi}
\]
is indistinguishable from the analogous histogram for type two zeta zeros (i.e.\ Figure \ref{F:4}); 
99.98\% are less than the average of 1 and 41.3\% are less than half the average.  


\subsubsection*{Acknowledgments}  We would like to thank both Rick Farr for sharing his computation of zeros of $\zeta^\prime(s)$ in the range $t<1000$, and David Farmer for sharing his computations in the range $10^6\le t\le 10^6+6\cdot 10^4$.  Thanks also to Fan Ge for suggesting a simpler argument with less explicit computation for Lemma \ref{Lemma:ZhangGe}, and helpful comments on the manuscript.  
And thanks to Cem  Y{\i}ld{\i}r{\i}m  for a very careful reading of the manuscript, and numerous helpful suggestions which greatly improved the exposition.


\begin{thebibliography}{99}
\bibitem{Akatsuka} H.\ Akatsuka, \emph{Conditional estimates for error terms related to the distribution of zeros of $\zeta^\prime(s)$}, J.\ Number Theory 132 (2012), pp.\  2242-2257.
\bibitem{Arias} J.\ Arias-de-Reyna, \emph{X-ray of Riemann's zeta function},  \href{https://arxiv.org/abs/math/0309433}{arXiv:math/0309433}
\bibitem{Conrey} J.\ B.\ Conrey, \emph{On the Real and Imaginary Curves fo the Riemann Zeta Function}, unpublished.
\bibitem{Duenez}  E.\ Due\~{n}ez, D.\ Farmer, S.\ Froehlich, C.P.\ Hughes, F.\ Mezzadri, T.\ Phan, \emph{Roots of the derivative of the Riemann-zeta function and of characteristic polynomials}, Nonlinearity 23 (2010), pp.\ 2599-2621.
\bibitem{FarmerKi} D.\ Farmer and H.\ Ki, \emph{Landau-Siegel zeros and zeros of the derivative of the
Riemann zeta function},  Advances in Math.\ 230 (2012) pp.\ 2048-2064.
\bibitem{Garaev} M.Z.\ Garaev and C.Y.\ Y{\i}ld{\i}r{\i}m , \emph{On Small Distances Between Ordinates of Zeros of $\zeta(s)$ and $\zeta^\prime(s)$}, IMRN (2007) no. 21.
\bibitem{FanGe1}  F.\ Ge, \emph{The number of zeros of $\zeta^\prime(s)$}, IMRN Vol.\ 2017, no.\ 5, pp.\ 1578-1588.
\bibitem{FanGe2} \bysame, \emph{The distribution of zeros of $\zeta'(s)$ and gaps between zeros of $\zeta(s)$}, Advances in Math.\  320 (2017), pp.\ 574-594.
\bibitem{GeGonek} F.\ Ge, S.\ Gonek, \emph{Critical points of polynomials with roots on the unit circle}, IMRN Vol.\ 2024, no. 7, pp.\  5434-5457.
\bibitem{Goldman} R.\ Goldman, \emph{Curvature formulas for implicit curves and surfaces}, Computer Aided Geometric Design 22 (2005), pp.\ 632-658.
\bibitem{Jerrard} R.\ Jerrard and L.\ Rubel, \emph{On the Curvature of the Level Lines of a Harmonic Function}, Proc.\ AMS, 14 (1963), pp.\ 29-32.
\bibitem{LandM} N.\ Levinson and H.\ Montgomery, \emph{Zeros of the derivatives of the Riemann zeta function}, Acta Math.\ 133 (1974), pp. 49-65.
\bibitem{Marden1} M.\ Marden, \emph{On the Zeros of the Derivative of an Entire Function}, MAA Monthly, 75 (1968), pp.\ 829-839.
\bibitem{Marden2} \bysame, \emph{On the derivative of an entire function}, Proc.\  Amer.\  Math.\  Soc.\  19 (1968), pp.\ 1045-1051.
\bibitem{Sou} K.\ Soundararajan, \emph{The horizontal distribution of zeros of $\zeta'(s)$}, Duke Math.\ J., 91 (1998), pp.\ 33-59.
\bibitem{Spira} R.\ Spira, \emph{Zero free regions of $\zeta^{(k)}(s)$}, Journal of the London Mathematical Society, 40 (1965), pp.\ 677-682.
\bibitem{Spira2}\bysame, \emph{Zeros of $\zeta^\prime(s)$ in the Critical Strip}, Proc.\ of the AMS, 35 (1972), pp.\ 59-60.
\bibitem{Stopple.Lehmer} J.\ Stopple, \emph{Lehmer Pairs Revisited}, Experimental Mathematics, 26 (2017), pp.\ 45-53.
\bibitem{Tit} E.\ Titchmarsh, The Theory of the Riemann Zeta Function, Oxford University Press, 2nd ed., 1986.
\bibitem{Zhang} Y.\ Zhang, \emph{On the zeros of $\zeta^\prime(s)$ near the critical line}, Duke J.\ Math.\ 110 (2001), pp.\ 555-572.
\end{thebibliography}
\end{document}